\documentclass[12pt]{amsart}
\usepackage{hyperref}
\usepackage{epsfig}
\usepackage{array}

\newtheorem{theorem}{Theorem}[section]
\newtheorem{definition}[theorem]{Definition}
\newtheorem{proposition}[theorem]{Proposition}

\newtheorem{lemma}[theorem]{Lemma}

\begin{document}

\title[Almost Hadamard matrices]
{Almost Hadamard matrices: general theory and examples}

\author{Teodor Banica}
\address{T.B.: Department of Mathematics, Cergy-Pontoise University, 95000 Cergy-Pontoise, France. {\tt teodor.banica@u-cergy.fr}}

\author{Ion Nechita}
\address{I.N.: CNRS, Laboratoire de Physique Th\'eorique, IRSAMC, Universit\'e de Toulouse, UPS, 31062 Toulouse, France. {\tt nechita@irsamc.ups-tlse.fr}}

\author{Karol \.Zyczkowski}
\address{K.Z.: Institute of Physics, Jagiellonian University, Cracow and Center for Theoretical Physics, Polish Academy of Sciences, Warsaw, Poland. {\tt karol@tatry.if.uj.edu.pl}}


\begin{abstract}
We develop a general theory of ``almost Hadamard matrices''. These are by definition the matrices $H\in M_N(\mathbb R)$ having the property that $U=H/\sqrt{N}$ is orthogonal, and is a local maximum of the $1$-norm on $O(N)$. Our study includes a detailed discussion of the circulant case ($H_{ij}=\gamma_{j-i}$) and of the two-entry case ($H_{ij}\in\{x,y\}$), with the construction of several families of examples, and some $1$-norm computations.
\end{abstract}

\maketitle

\section*{Introduction}

An Hadamard matrix is a square matrix having $\pm 1$ entries, whose rows are pairwise orthogonal. The simplest example, appearing at $N=2$, is the Walsh matrix:
$$H_2=\begin{pmatrix}1&1\\ 1&-1\end{pmatrix}$$

At $N=3$ we cannot have examples, due to the orthogonality condition, which forces $N$ to be even. At $N=4$ now, we have several examples, for instance $H_4=H_2\otimes H_2$: 
$$H_4=\begin{pmatrix}1&1&1&1\\ 1&-1&1&-1\\ 1&1&-1&-1\\ 1&-1&-1&1\end{pmatrix}$$

For higher values of $N$, the construction of Hadamard matrices is quite a tricky problem. First, by permuting rows and columns or by multiplying them by $-1$, we can always assume that the first 3 rows of our matrix look as follows:
$$H=\begin{pmatrix}
1\ldots 1&1\ldots 1&1\ldots 1&1\ldots 1\\
1\ldots 1&1\ldots 1&-1\ldots -1&-1\ldots -1\\
1\ldots 1&-1\ldots -1&1\ldots 1&-1\ldots -1\\
\ldots&\ldots&\ldots&\ldots
\end{pmatrix}$$

Now if we denote by $x,y,z,t$ the sizes of the 4 columns, the orthogonality conditions between the first 3 rows give $x=y=z=t$, so $N=x+y+z+t$ is a multiple of 4.

A similar analysis with 4 rows instead of 3, or any other kind of abstract or concrete consideration doesn't give any further restriction on $N$, and we have:

\medskip

\noindent {\bf Hadamard Conjecture (HC).} {\em There is at least one Hadamard matrix of size $N\times N$, for any $N\in 4\mathbb N$.}

\medskip

This conjecture, going back to the 19th century, is probably one of the most beautiful statements in combinatorics, and in mathematics in general. The numeric verification so far goes up to $N=664$, see \cite{kta}. For a general presentation of the problem, see \cite{lgo}.

At the level of concrete examples, the only ones which are simple to describe are the tensor powers of the Walsh matrix, having size $2^k$. For some other examples, see \cite{hor}.

Yet another good problem, simple-looking as well, concerns the circulant case. Given a vector $\gamma\in(\pm 1)^N$, one can ask whether the matrix $H\in M_N(\pm 1)$ defined by $H_{ij}=\gamma_{j-i}$ is Hadamard or not. Here is a solution to the problem, appearing at $N=4$:
$$K_4=\begin{pmatrix}-1&1&1&1\\ 1&-1&1&1\\ 1&1&-1&1\\ 1&1&1&-1\end{pmatrix}$$

More generally, any vector $\gamma\in(\pm 1)^4$ satisfying $\sum\gamma_i=\pm 1$ is a solution to the problem. The following conjecture, going back to \cite{rys}, states that there are no other solutions:

\medskip

\noindent {\bf Circulant Hadamard Conjecture (CHC).} {\em There is no circulant Hadamard matrix of size $N\times N$, for any $N\neq 4$.}

\medskip

The fact that such a simple-looking problem is still open might seem quite surprizing. If we denote by $S\subset\{1,\ldots,N\}$ the set of positions of the $-1$ entries of $\gamma$, the Hadamard matrix condition is simply $|S\cap(S+k)|=|S|-N/4$, for any $k\neq 0$, taken modulo $N$. Thus, the above conjecture simply states that at $N\neq 4$, such a set $S$ cannot exist!

Summarizing, the Hadamard matrices are very easy to introduce, and they quickly lead to some difficult and interesting combinatorial problems. Regarding now the other motivations for studying such matrices, these are quite varied:
\begin{enumerate}
\item The Hadamard matrices were first studied by Sylvester \cite{syl}, who was seemingly attracted by their plastic beauty: just replace the $\pm 1$ entries by black and white tiles, and admire the symmetries and dissymmetries of the resulting design!

\item More concretely now, the Hadamard matrices can be used for various coding purposes, and have several applications to engineering, and quantum physics. For instance the Walsh matrices $H_2^{\otimes k}$ are used in the Olivia MFSK radio protocol.
\end{enumerate}

Most applications of the Hadamard matrices, however, come from their generalizations. A ``complex Hadamard'' matrix is a matrix $H\in M_N(\mathbb C)$ all whose entries are on the unit circle, and whose rows are pairwise orthogonal. The basic example is $\widetilde{F}_N=\sqrt{N}F_N$, where $F_N$ is the matrix of the Fourier transform over $\mathbb Z_N$. That is, with $\omega=e^{2\pi i/N}$:
$$\widetilde{F}_N=\begin{pmatrix}
1&1&1&\ldots&1\\
1&\omega&\omega^2&\ldots&\omega^{N-1}\\
1&\omega^2&\omega^4&\ldots&\omega^{2(N-1)}\\
\ldots&\ldots&\ldots&\ldots&\ldots\\
1&\omega^{N-1}&\omega^{2(N-1)}&\ldots&\omega^{(N-1)^2}
\end{pmatrix}$$

As a first observation, the existence of this matrix prevents the existence of a ``complex version'' of the HC. However, when trying to construct complex Hadamard matrices by using roots of unity of a given order, a wide, subtle, and quite poorly understood generalization of the HC problematics appears. See e.g. \cite{bbs}, \cite{but}, \cite{lle}, \cite{lau}.

As for the motivations and applications, these partly come from pure mathematics, cf. e.g. \cite{ba2}, \cite{haa}, \cite{jon}, \cite{pop}, and partly come from quantum physics, cf. e.g. \cite{bb+}, \cite{tzy}.

Let us go back now to the real case. Since the determinant of $N$ vectors is maximized when these vectors are chosen pairwise orthogonal, we have the following result:

\medskip

\noindent {\bf Theorem A.} {\em For a matrix $H\in M_N(\pm 1)$ we have $|\det H|\leq N^{N/2}$, with equality if and only if $H$ is Hadamard.}

\medskip

This result, due to Hadamard himself \cite{had}, has led to a number of interesting problems, and to the general development of the theory of Hadamard matrices. See \cite{hor}.

As already mentioned, in order for an Hadamard matrix to exist, its size $N$ must be a multiple of $4$. For numbers of type $N=4n+k$ with $k=1,2,3$, several ``real'' generalizations of the Hadamard matrices have been constructed. The idea is usually to consider matrices $H\in M_N(\pm 1)$, whose rows are as orthogonal as they can be:

\medskip

\noindent {\bf Definition A.} {\em A ``quasi-Hadamard'' matrix is a square matrix $H\in M_N(\pm 1)$ which is as orthogonal as possible, e.g. which maximizes the quantity $|det H|$.}

\medskip

This definition is of course a bit vague, but the main idea is there. For a detailed discussion of the different notions here, we refer to the articles \cite{aa+}, \cite{kko}, \cite{pso}. 

Yet another interpretation of the Hadamard matrices comes from the Cauchy-Schwarz inequality. Since for $U\in O(N)$ we have $||U||_2=\sqrt{N}$, we obtain:

\medskip

\noindent {\bf Theorem B.} {\em For a matrix $U\in O(N)$ we have $||U||_1\leq N\sqrt{N}$, with equality if and only if $H=\sqrt{N}U$ is Hadamard.}

\medskip

This result, first pointed out in \cite{bc1}, shows that the matrices of type $H=\sqrt{N}U$, with $U\in O(N)$ being a maximizer of the 1-norm on $O(N)$, can be thought of as being some kind of ``analytic generalizations'' of the Hadamard matrices. Note that such matrices exist for any $N$, in particular for $N=4n+k$ with $k=1,2,3$.

As an example, the maximum of the 1-norm on $O(3)$ can be shown to be the number $5$, coming from the matrix $U_3=K_3/\sqrt{3}$ and its various conjugates, where:
$$K_3=\frac{1}{\sqrt{3}}\begin{pmatrix}-1&2&2\\ 2&-1&2\\ 2&2&-1\end{pmatrix}$$

This result, proved in \cite{bc1} by using the Euler-Rodrigues formula, is actually something quite accidental.  In general, the integration on $O(N)$ is quite a subtle business, and the maximum of the 1-norm is quite difficult to approach. See \cite{bc2}, \cite{bsc}, \cite{csn}. 

These integration problems make the above type of matrices quite hard to investigate. Instead of looking directly at them, we will rather enlarge the attention to the matrices of type $H=\sqrt{N}U$, where $U\in O(N)$ is a local maximizer of the 1-norm on $O(N)$. Indeed, according to the Hessian computations in \cite{bc1}, these latter matrices are characterized by the fact that all their entries are nonzero, and $SU^t>0$, where $S_{ij}={\rm sgn}(U_{ij})$.

Summarizing, Theorem B suggests the following definition:

\medskip

\noindent {\bf Definition B.} {\em An ``almost Hadamard'' matrix is a square matrix $H\in M_N(\mathbb R)$ such that $U=H/\sqrt{N}$ is orthogonal, and is a local maximum of the $1$-norm on $O(N)$.}

\medskip

The basic examples are of course the Hadamard matrices. We have as well the following $N\times N$ matrix, with $N\in\mathbb N$ arbitrary, which generalizes the above matrices $K_3,K_4$:
$$K_N=\frac{1}{\sqrt{N}}\begin{pmatrix}
2-N&2&\ldots&2&2\\
2&2-N&\ldots&\ldots&2\\
\ldots&\ldots&\ldots&\ldots&\ldots\\
2&\ldots&\ldots&\ldots&\ldots\\
2&2&\ldots&\ldots&2-N
\end{pmatrix}$$

A lot of other interesting examples exist, as we will show in this paper. Here is for instance a remarkable matrix, having order $N\in 2\mathbb N+1$, and circulant structure:
\setlength{\extrarowheight}{8pt}
$$L_N=\frac{1}{\sqrt{N}}
\begin{pmatrix}
1&-\cos^{-1}\frac{\pi}{N}&\cos^{-1}\frac{2\pi}{N}&\ldots&\cos^{-1}\frac{(N-1)\pi}{N}\\
\cos^{-1}\frac{(N-1)\pi}{N}&1&-\cos^{-1}\frac{\pi}{N}&\ldots&-\cos^{-1}\frac{(N-2)\pi}{N}\\
\ldots&\ldots&\ldots&\ldots&\ldots\\
-\cos^{-1}\frac{\pi}{N}&\cos^{-1}\frac{2\pi}{N}&-\cos^{-1}\frac{3\pi}{N}&\ldots&1
\end{pmatrix}$$
\setlength{\extrarowheight}{0pt}

Yet another series, with $N=q^2+q+1$, where $q=p^k$ is a prime power, comes from the adjacency matrix of the projective plane over $\mathbb F_q$. Here is for instance the matrix associated to the Fano plane ($q=2$), where $x=2-4\sqrt{2}$, $y=2+3\sqrt{2}$:
$$I_7=\frac{1}{2\sqrt{7}}\begin{pmatrix}
x&x&y&y&y&x&y\\
y&x&x&y&y&y&x\\
x&y&x&x&y&y&y\\
y&x&y&x&x&y&y\\
y&y&x&y&x&x&y\\
y&y&y&x&y&x&x\\
x&y&y&y&x&y&x
\end{pmatrix}$$

The aim of the present paper is to provide a systematic study of such matrices, with the construction of a number of non-trivial examples, and with the development of some general theory as well. Our motivation comes from two kinds of problems:
\begin{enumerate}
\item The Hadamard matrix problematics. The world of Hadamard matrices is extremely rigid, and we think that our study of almost Hadamard matrices, where there is much more freedom, can help. As an example, there are several non-trivial classes of circulant almost Hadamard matrices at any $N\in\mathbb N$, and trying to understand them might end up in sheding some new light on the CHC.

\item Generalizations of Hadamard matrices. The Hadamard matrices have applications in a number of areas of physics and engineering, notably in coding theory, and in various branches of quantum physics. One problem, however, is that these matrices exist only at $N=4n$. At $N=4n+k$ with $k=1,2,3$ some generalizations would be needed, and we believe that our almost Hadamard matrices can help.
\end{enumerate}

The paper is organized as follows: 1 is a preliminary section, in 2-3 we investigate two special cases, namely the circulant case and the two-entry case, and 4 contains a list of examples. The final section, 5, contains a few concluding remarks.

\bigskip

\noindent {\bf Acknowledgements.} We would like to thank Guillaume Aubrun for several useful discussions. The work of T.B. was supported by the ANR grant ``Granma''. I.N. acknowledges financial support from the ANR project OSvsQPI 2011 BS01 008 01 and from a CNRS PEPS grant. The work of K.Z. was supported by the grant N N202 261938, financed by the Polish Ministry of Science and Higher Education.

\section{Preliminaries}

We consider in this paper various square matrices $M\in M_N(\mathbb C)$. The indices of our matrices will usually range in the set $\{0,1,\ldots,N-1\}$.

\begin{definition}
We use the following special $N\times N$ matrices: 
\begin{enumerate}
\item $1_N$ is the identity matrix.

\item $J_N$ is the ``flat'' matrix, having all entries equal to $1/N$.

\item $F_N$ is the Fourier matrix, given by $(F_N)_{ij}=\omega^{ij}/\sqrt{N}$, with $\omega=e^{2\pi i/N}$.
\end{enumerate}
\end{definition}

We denote by $D$ the generic diagonal matrices, by $U$ the generic orthogonal or unitary matrices, and by $H$ the generic Hadamard matrices, and their generalizations.

Our starting point is the following observation, from \cite{bc1}:

\begin{proposition}
For $U\in O(N)$ we have $||U||_1\leq N\sqrt{N}$, with equality if and only if $H=\sqrt{N}U$ is Hadamard.
\end{proposition}

\begin{proof}
The first assertion follows from the Cauchy-Schwarz inequality:
$$\sum_{ij}|U_{ij}|\leq\left(\sum_{ij}1^2\right)^{1/2}\left(\sum_{ij}|U_{ij}|^2\right)^{1/2}=N\sqrt{N}$$

For having equality the numbers $|U_{ij}|$ must be equal, and since the sum of squares of these numbers is $N$, we must have $|U_{ij}|=1/\sqrt{N}$, which gives the result.
\end{proof}

As already mentioned in the introduction, the study of maximizers of the 1-norm on $O(N)$ is a quite difficult task. However, we have here the following result, from \cite{bc1}:

\begin{theorem}
For $U\in O(N)$, the following are equivalent:
\begin{enumerate}
\item $U$ is a local maximizer of the $1$-norm on $O(N)$.

\item $U_{ij}\neq 0$, and $SU^t>0$, where $S_{ij}={\rm sgn}(U_{ij})$.
\end{enumerate}
\end{theorem}

\begin{proof}
As already mentioned, this result is from \cite{bc1}. Here is the idea of the proof:

Let us first prove that if $U$ is a local maximizer of the 1-norm, then $U_{ij}\neq 0$. For this purpose, let $U_1,\ldots,U_N$ be the columns of $U$, and let us perform a rotation of $U_1,U_2$:
$$\begin{pmatrix}U^t_1\\ U^t_2\end{pmatrix}=\begin{pmatrix}
\cos t\cdot U_1-\sin t\cdot U_2\\ \sin t\cdot U_1+\cos t\cdot U_2
\end{pmatrix}$$

In order to compute the 1-norm, let us permute the columns of $U$, in such a way that the first two rows look as follows, with $X_k\neq 0$, $Y_k\neq 0$, $A_kC_k>0$, $B_kD_k<0$:
$$\begin{pmatrix}U_1\\ U_2\end{pmatrix}
=\begin{pmatrix}
0&0&Y&A&B\\
0&X&0&C&D
\end{pmatrix}$$

If we agree that the lower-case letters denote the 1-norms of the corresponding upper-case vectors, and we let $K=u_3+\ldots+u_N$, then for $t>0$ small we have:
\begin{eqnarray*}
||U^t||_1
&=&||\cos t\cdot U_1-\sin t\cdot U_2||_1+||\sin t\cdot U_1+\cos t\cdot U_2||_1+K\\
&=&(\cos t+\sin t)(x+y+a+d)+(\cos t-\sin t)(b+c)+K
\end{eqnarray*}

Now since $U$ locally maximizes the 1-norm on $O(N)$, the derivative of this quantity must be negative in the limit $t\to 0$. So, we obtain the following inequality:
$$(x+y+a+d)-(b+c)\leq 0$$

Consider now the matrix obtained by interchanging $U_1,U_2$. Since this matrix must be as well a local maximizer of the 1-norm on $O(N)$, we obtain:
$$(x+y+b+c)-(a+d)\leq 0$$

We deduce that $x+y=0$, so $x=y=0$, and the $0$ entries of $U_1,U_2$ must appear at the same positions. By permuting the rows of $U$ the same must hold for any two rows $U_i,U_j$. Now since $U$ cannot have zero columns, all its entries must be nonzero, as claimed.

It remains to prove that if $U_{ij}\neq 0$, then $U$ is a local maximizer of $F(U)=||U||_1$ if and only if $SU^t>0$, where $S_{ij}={\rm sgn}(U_{ij})$. For this purpose, we differentiate $F$:
$$dF=\sum_{ij}S_{ij}dU_{ij}$$

We know that $O(N)$ consists of the zeroes of the polynomials $A_{ij}=\sum_kU_{ik}U_{jk}-\delta_{ij}$. So, $U$ is a critical point of $F$ if and only if $dF\in span(dA_{ij})$. Now since $A_{ij}=A_{ji}$, this is the same as asking for a symmetric matrix $M$ such that $dF=\sum_{ij}M_{ij}dA_{ij}$. But:
$$\sum_{ij}M_{ij}dA_{ij}
=\sum_{ijk}M_{ij}(U_{ik}dU_{jk}+U_{jk}dU_{ik})
=2\sum_{lk}(MU)_{lk}dU_{lk}$$

Thus the critical point condition reads $S=2MU$, so the matrix $M=SU^t/2$ must be symmetric. Now the Hessian of $F$ applied to a vector $X=UY$, with $Y\in O(N)$, is:
$$Hess(F)(X)
=\frac{1}{2}Tr(X^t\cdot SU^t\cdot X)
=\frac{1}{2}Tr(Y^t\cdot U^tS\cdot Y)$$

Thus the Hessian of $F$ is positive definite when $U^tS$ is positive definite, which is the same as saying that $U(U^tS)U^t=SU^t$ is positive definite, and we are done. 
\end{proof}

The above result gives rise to the following definition:

\begin{definition}
A square matrix $H\in M_N(\mathbb R^*)$ is called ``almost Hadamard'' if $U=H/\sqrt{N}$ is orthogonal, and the following equivalent conditions are satisfied:
\begin{enumerate}
\item $U$ is a local maximizer of the $1$-norm on $O(N)$.

\item $U_{ij}\neq 0$, and with $S_{ij}={\rm sgn}(U_{ij})$, we have $SU^t>0$.
\end{enumerate}
If so is the case, we call $H$ ``optimal'' if $U$ is a maximizer of the $1$-norm on $O(N)$.
\end{definition}

Let $J_N$ be the flat $N\times N$ matrix, having all the entries equal to $1/N$. Also, let us call ``Hadamard equivalence'', or just ``equivalence'', the equivalence relation on the $N\times N$ matrices coming from permuting the rows and columns, or multiplying them by $-1$.

\begin{proposition}
The class of almost Hadamard matrices has the following properties:
\begin{enumerate}
\item It contains all the Hadamard matrices.

\item It contains the matrix $K_N=\sqrt{N}(2J_N-1_N)$.

\item It is stable under equivalence, tensor products, and transposition.
\end{enumerate}
\end{proposition}

\begin{proof}
All the assertions are clear from definitions:

(1) This follows either from Proposition 1.2, or from the fact that $U=H/\sqrt{N}$ is orthogonal, and $SU^t=HU^t=\sqrt{N}1_N$ is positive.

(2) First, the matrix $U=K_N/\sqrt{N}$ is orthogonal, because it is symmetric, and:
$$U^2=(2J_N-1_N)^2=4J_N^2-4J_N+1_N=1_N$$

Also, we have $S=NJ_N-21_N$, so the matrix $SU^t$ is indeed positive:
$$SU^t=(NJ_N-21_N)(2J_N-1_N)=(N-2)J_N+2(1_N-J_N)$$

(3) For a tensor product of almost Hadamard matrices $H=H'\otimes H''$ we have $U=U'\otimes U''$ and $S=S'\otimes S''$, so that $U$ is unitary and $SU^t$ is positive, as claimed. As for the assertions regarding equivalence and transposition, these are clear from definitions.
\end{proof}

Regarding now the optimal case, we have the following result, from \cite{bc1}:

\begin{proposition}
The optimal almost Hadamard matrices are as follows:
\begin{enumerate}
\item At any $N$ where HC holds, these are the $N\times N$ Hadamard matrices.

\item At $N=3$, these are precisely $K_3=\sqrt{3}(2J_3-1_3)$ and its conjugates.
\end{enumerate}
\end{proposition}

\begin{proof}
The assertion (1) is clear from Proposition 1.2. For (2) we must prove that for $U\in O(3)$ we have $||U||_1\leq 5$, with equality when $U$ is conjugate to $U_3=2J_3-1_3$. But here we can assume $U\in SO(3)$, and use the Euler-Rodrigues formula:
$$U=\begin{pmatrix}
x^2+y^2-z^2-t^2&2(yz-xt)&2(xz+yt)\\
2(xt+yz)&x^2+z^2-y^2-t^2&2(zt-xy)\\
2(yt-xz)&2(xy+zt)&x^2+t^2-y^2-z^2
\end{pmatrix}$$

Here $(x,y,z,t)\in S^3$ comes from the standard cover map $S^3\simeq SU(2)\to SO(3)$. Now by linearizing, we must prove that for any $(x,y,z,t)\in\mathbb R^4$ we have:
$$||U||_1\leq 5(x^2+y^2+z^2+t^2)$$

The proof of this latter inequality is routine, and the equality situation turns to hold indeed exactly for the matrix $U_3=2J_3-1_3$ and its conjugates. See \cite{bc1}.
\end{proof}

Finally, let us mention that a version of Proposition 1.2 above, using the H\"older inequality, shows that the matrices of type $U=H/\sqrt{N}$ with $H\in M_N(\pm 1)$ Hadamard maximize the $p$-norm on $O(N)$ at $p\in [1,2)$, and minimize it at $p\in(2,\infty]$. See \cite{bc1}. Part of the above $p=1$ results extend to the general setting $p\in [1,\infty]-\{2\}$, and in particular to the exponents $p=4$ and $p=\infty$, which are of particular interest in connection with several quantum physics questions. This will be discussed in a forthcoming paper.

\section{The circulant case}

In this section we study the almost Hadamard matrices which are circulant. We recall that a matrix $H\in M_N(\mathbb C)$ is called circulant if it is of the form:
$$H=
\begin{pmatrix}
\gamma_0&\gamma_1&\ldots&\gamma_{N-1}\\
\gamma_{N-1}&\gamma_0&\ldots&\gamma_{N-2}\\
\ldots&\ldots&\ldots&\ldots\\
\gamma_1&\gamma_2&\ldots&\gamma_0
\end{pmatrix}$$

Let $F\in U(N)$ be the Fourier matrix, given by $F_{ij}=\omega^{ij}/\sqrt{N}$, where $\omega=e^{2\pi i/N}$. Given a vector $\alpha\in\mathbb C^n$, we associate to it the diagonal matrix $\alpha'=diag(\alpha_0,\ldots,\alpha_{N-1})$.

We will make a heavy use of the following well-known result:

\begin{proposition}
For a matrix $H\in M_N(\mathbb C)$, the following are equivalent:
\begin{enumerate}
\item $H$ is circulant, i.e. $H_{ij}=\gamma_{j-i}$, for a certain vector $\gamma\in\mathbb C^N$.

\item $H$ is Fourier-diagonal, i.e. $H=FDF^*$, with $D\in M_N(\mathbb C)$ diagonal.
\end{enumerate}
In addition, if so is the case, then with $D=\sqrt{N}\alpha'$ we have $\gamma=F\alpha$.
\end{proposition}

\begin{proof}
(1)$\implies$(2) The matrix $D=F^*HF$ is indeed diagonal, given by:
$$D_{ij}=\frac{1}{N}\sum_{kl}\omega^{jl-ik}\gamma_{l-k}=\delta_{ij}\sum_r\omega^{jr}\gamma_r$$ 

(2)$\implies$(1) The matrix $H=FDF^*$ is indeed circulant, given by:
$$H_{ij}=\sum_kF_{ik}D_{kk}\bar{F}_{jk}=\frac{1}{N}\sum_k\omega^{(i-j)k}D_{kk}$$

Finally, the last assertion is clear from the above formula of $H_{ij}$.
\end{proof}

Let us investigate now the circulant orthogonal matrices. We let the matrix indices $i,j$ vary modulo $N$. We denote by $\mathbb T$ the unit circle in the complex plane.

\begin{lemma}
For a matrix $U\in M_N(\mathbb C)$, the following are equivalent:
\begin{enumerate}
\item $U$ is orthogonal and circulant.

\item $U=F\alpha'F^*$ with $\alpha\in\mathbb T^N$ satisfying $\bar{\alpha}_i=\alpha_{-i}$ for any $i$.
\end{enumerate}
\end{lemma}

\begin{proof}
We will use many times the fact that given $\alpha\in\mathbb C^N$, the vector $\gamma=F\alpha$ is real if and only if $\bar{\alpha}_i=\alpha_{-i}$ for any $i$. This follows indeed from $\overline{F\alpha}=F\tilde{\alpha}$, with $\tilde{\alpha}_i=\bar{\alpha}_{-i}$.

(1)$\implies$(2) Write $H_{ij}=\gamma_{j-i}$ with $\gamma\in\mathbb R^N$. By using Proposition 2.1 we obtain $H=FDF^*$ with $D=\sqrt{N}\alpha'$ and $\gamma=F\alpha$. Now since $U=F\alpha'F^*$ is unitary, so is $\alpha'$, so we must have $\alpha\in\mathbb T^N$. Finally, since $\gamma$ is real we have $\bar{\alpha}_i=\alpha_{-i}$, and we are done.

(2)$\implies$(1) We know from Proposition 2.1 that $U$ is circulant. Also, from $\alpha\in\mathbb T^N$ we obtain that $\alpha'$ is unitary, and so must be $U$. Finally, since we have $\bar{\alpha}_i=\alpha_{-i}$, the vector $\gamma=F\alpha$ is real, and hence we have $U\in M_N(\mathbb R)$, which finishes the proof.
\end{proof}

Let us discuss now the almost Hadamard case. First, in the usual Hadamard case, the known examples and the corresponding $\alpha$-vectors are as follows:

\begin{proposition}
The known circulant Hadamard matrices, namely
$$\pm\begin{pmatrix}
-1\!\!&\!\!1\!\!&\!\!1\!\!&\!\!1\\
1\!\!&\!\!-1\!\!&\!\!1\!\!&\!\!1\\
1\!\!&\!\!1\!\!&\!\!-1\!\!&\!\!1\\
1\!\!&\!\!1\!\!&\!\!1\!\!&\!\!-1
\end{pmatrix},
\pm\begin{pmatrix}
1\!\!&\!\!-1\!\!&\!\!1\!\!&\!\!1\\
1\!\!&\!\!1\!\!&\!\!-1\!\!&\!\!1\\
1\!\!&\!\!1\!\!&\!\!1\!\!&\!\!-1\\
-1\!\!&\!\!1\!\!&\!\!1\!\!&\!\!1
\end{pmatrix},
\pm\begin{pmatrix}
1\!\!&\!\!1\!\!&\!\!-1\!\!&\!\!1\\
1\!\!&\!\!1\!\!&\!\!1\!\!&\!\!-1\\
-1\!\!&\!\!1\!\!&\!\!1\!\!&\!\!1\\
1\!\!&\!\!-1\!\!&\!\!1\!\!&\!\!1
\end{pmatrix},
\pm\begin{pmatrix}
1\!\!&\!\!1\!\!&\!\!1\!\!&\!\!-1\\
-1\!\!&\!\!1\!\!&\!\!1\!\!&\!\!1\\
1\!\!&\!\!-1\!\!&\!\!1\!\!&\!\!1\\
1\!\!&\!\!1\!\!&\!\!-1\!\!&\!\!1
\end{pmatrix}$$
come from the vectors $\alpha=\pm(1,-1,-1,-1),\pm(1,-i,1,i),\pm(1,1,-1,1),\pm(1,i,1,-i)$.
\end{proposition}

\begin{proof}
At $N=4$ the conjugate of the Fourier matrix is given by:
$$F^*=\frac{1}{2}\begin{pmatrix}
1&1&1&1\\
1&-i&-1&i\\
1&-1&1&-1\\
1&i&-1&-i
\end{pmatrix}$$

Thus the vectors $\alpha=F^*\gamma$ are indeed those in the statement.
\end{proof}

We have the following ``almost Hadamard'' generalization of the above matrices:

\begin{proposition}
If $q^N=1$ then the vector $\alpha=\pm(1,-q,-q^2,\ldots,-q^{N-1})$ produces an almost Hadamard matrix, which is equivalent to $K_N=\sqrt{N}(2J_N-1_N)$. 
\end{proposition}

\begin{proof}
Observe first that these matrices generalize those in Proposition 2.3. Indeed, at $N=4$ the choices for $q$ are $1,i,-1,-i$, and this gives the above $\alpha$-vectors.

Assume that the $\pm$ sign in the statement is $+$. With $q=\omega^r$, we have:
$$\sqrt{N}\gamma_i=\sum_{k=0}^{N-1}\omega^{ik}\alpha_k=1-\sum_{k=1}^{N-1}\omega^{(i+r)k}=2-\sum_{k=0}^{N-1}\omega^{(i+r)k}=2-\delta_{i,-r}N$$

In terms of the standard long cycle $(C_N)_{ij}=\delta_{i+1,j}$, we obtain:
$$H=\sqrt{N}(2J_N-C_N^{-r})$$

Thus $H$ is equivalent to $K_N$, and by Proposition 1.5, it is almost Hadamard.
\end{proof}

In general, the construction of circulant almost Hadamard matrices is quite a tricky problem. At the abstract level, we have the following technical result:

\begin{lemma}
A circulant matrix $H\in M_N(\mathbb R^*)$, written $H_{ij}=\gamma_{j-i}$, is almost Hadamard if and only if the following conditions are satisfied:
\begin{enumerate}
\item The vector $\alpha=F^*\gamma$ satisfies $\alpha\in\mathbb T^N$.

\item With $\varepsilon={\rm sgn}(\gamma)$, $\rho_i=\sum_r\varepsilon_r\gamma_{i+r}$ and $\nu=F^*\rho$, we have $\nu>0$.
\end{enumerate}
In addition, if so is the case, then $\bar{\alpha}_i=\alpha_{-i}$, $\rho_i=\rho_{-i}$ and $\nu_i=\nu_{-i}$ for any $i$.
\end{lemma}

\begin{proof}
According to Definition 1.4 our matrix $H$ is almost Hadamard if any only if the matrix $U=H/\sqrt{N}$ is orthogonal and $SU^t>0$, where $S_{ij}={\rm sgn}(U_{ij})$. By Lemma 2.2 the orthogonality of $U$ is equivalent to the condition (1). Regarding now the condition $SU^t>0$, this is equivalent to $S^tU>0$. But, with $k=i-r$, we have:
$$(S^tH)_{ij}=\sum_kS_{ki}H_{kj}=\sum_k\varepsilon_{i-k}\gamma_{j-k}=\sum_r\varepsilon_r\gamma_{j-i+r}=\rho_{j-i}$$

Thus $S^tU$ is circulant, with $\rho/\sqrt{N}$ as first row. From Proposition 2.1 we get $S^tU=FLF^*$ with $L=\nu'$ and $\nu=F^*\rho$, so $S^tU>0$ iff $\nu>0$, which is the condition (2).

Finally, the assertions about $\alpha,\nu$ follow from the fact that $F\alpha,F\nu$ are real. As for the assertion about $\rho$, this follows from the fact that $S^tU$ is symmetric.
\end{proof}

\begin{theorem}
For $N$ odd the following matrix is almost Hadamard,
\setlength{\extrarowheight}{8pt}
$$L_N=\frac{1}{\sqrt{N}}
\begin{pmatrix}
1&-\cos^{-1}\frac{\pi}{N}&\cos^{-1}\frac{2\pi}{N}&\ldots&\cos^{-1}\frac{(N-1)\pi}{N}\\
\cos^{-1}\frac{(N-1)\pi}{N}&1&-\cos^{-1}\frac{\pi}{N}&\ldots&-\cos^{-1}\frac{(N-2)\pi}{N}\\
\ldots&\ldots&\ldots&\ldots&\ldots\\
-\cos^{-1}\frac{\pi}{N}&\cos^{-1}\frac{2\pi}{N}&-\cos^{-1}\frac{3\pi}{N}&\ldots&1
\end{pmatrix}$$
\setlength{\extrarowheight}{0pt}
and comes from an $\alpha$-vector having all entries equal to $1$ or $-1$.
\end{theorem}

\begin{proof}
Write $N=2n+1$, and consider the following vector:
$$\alpha_i=\begin{cases}
(-1)^{n+i}&{\rm for }\ i=0,1,\ldots,n\\
(-1)^{n+i+1}&{\rm for}\ i=n+1,\ldots,2n
\end{cases}$$

Let us first prove that $(L_N)_{ij}=\gamma_{j-i}$, where $\gamma=F\alpha$. With $\omega=e^{2\pi i/N}$ we have:
$$\sqrt{N}\gamma_i=\sum_{j=0}^{2n}\omega^{ij}\alpha_j=\sum_{j=0}^n(-1)^{n+j}\omega^{ij}+\sum_{j=1}^n(-1)^{n+(N-j)+1}\omega^{i(N-j)}$$

Now since $N$ is odd, and since $\omega^N=1$, we obtain:
$$\sqrt{N}\gamma_i=\sum_{j=0}^n(-1)^{n+j}\omega^{ij}+\sum_{j=1}^n(-1)^{n-j}\omega^{-ij}=\sum_{j=-n}^n(-1)^{n+j}\omega^{ij}$$

By computing the sum on the right, with $\xi=e^{\pi i/N}$ we get, as claimed:
$$\sqrt{N}\gamma_i=\frac{2\omega^{-ni}}{1+\omega^i}=\frac{2\xi^{-2ni}}{1+\xi^{2i}}=\frac{2\xi^{-Ni}}{\xi^{-i}+\xi^i}=(-1)^i\cos^{-1}\frac{i\pi}{N}$$

In order to prove now that $L_N$ is almost Hadamard, we use Lemma 2.5. Since the sign vector is simply $\varepsilon=(-1)^n\alpha$, the vector $\rho_i=\sum_r\varepsilon_r\gamma_{i+r}$ is given by:
$$\sqrt{N}\rho_i=(-1)^n\sum_{r=0}^{2n}\alpha_r\sum_{j=-n}^n(-1)^{n+j}\omega^{(i+r)j}=\sum_{j=-n}^n(-1)^j\omega^{ij}\sum_{r=0}^{2n}\alpha_r\omega^{rj}$$

Now since the last sum on the right is $(\sqrt{N}F\alpha)_j=\sqrt{N}\gamma_j$, we obtain:
$$\rho_i=\sum_{j=-n}^n(-1)^j\omega^{ij}\gamma_j=\frac{1}{\sqrt{N}}\sum_{j=-n}^n(-1)^j\omega^{ij}\sum_{k=-n}^n(-1)^{n+k}\omega^{jk}$$

Thus we have the following formula:
$$\rho_i=\frac{(-1)^n}{\sqrt{N}}\sum_{j=-n}^n\sum_{k=-n}^n(-1)^{j+k}\omega^{(i+k)j}$$

Let us compute now the vector $\nu=F^*\rho$. We have:
$$\nu_l=\frac{1}{\sqrt{N}}\sum_{i=0}^{2n}\omega^{-il}\rho_i=\frac{(-1)^n}{N}\sum_{j=-n}^n\sum_{k=-n}^n(-1)^{j+k}\omega^{jk}\sum_{i=0}^{2n}\omega^{i(j-l)}$$

The sum on the right is $N\delta_{jl}$, with both $j,l$ taken modulo $N$, so it is equal to $N\delta_{jL}$, where $L=l$ for $l\leq n$, and $L=l-N$ for $l>n$. We get:
$$\nu_l=(-1)^n\sum_{k=-n}^n(-1)^{L+k}\omega^{Lk}=(-1)^{n+L}\sum_{k=-n}^n(-w^L)^k$$

With $\xi=e^{\pi i/N}$, this gives the following formula:
$$\nu_l=(-1)^{n+L}\frac{2(-\omega^L)^{-n}}{1+\omega^L}=(-1)^L\frac{2\omega^{-nL}}{1+\omega^L}$$

In terms of the variable $\xi=e^{\pi i/N}$, we obtain:
$$\nu_l=(-1)^L\frac{2\xi^{-2nL}}{1+\xi^{2L}}=(-1)^L\frac{2\xi^{-NL}}{\xi^{-L}+\xi^L}=\cos^{-1}\frac{L\pi}{N}$$

Now since $L\in[-n,n]$, all the entries of $\nu$ are positive, and we are done.
\end{proof}

At the level of examples now, at $N=3$ we obtain the matrix $L_3=-K_3$. At $N=5$ we obtain a matrix having as entries 1 and $x=-\cos^{-1}\frac{\pi}{5}$, $y=\cos^{-1}\frac{2\pi}{5}$:
$$L_5=\frac{1}{\sqrt{5}}\begin{pmatrix}
1&x&y&y&x\\
x&1&x&y&y\\
y&x&1&x&y\\
y&y&x&1&x\\
x&y&y&x&1
\end{pmatrix}$$

Let us look now more in detail at the vectors $\alpha\in\mathbb T^N$ appearing in Proposition 2.4 and in the proof of Theorem 2.6. In both cases we have $\alpha_i^2=\omega^{ri}$ for a certain $r\in\mathbb N$, and this might suggest that any circulant almost Hadamard matrix should come from a vector $\alpha\in\mathbb T^N$ having the property that $\alpha^2$ is formed by roots of unity in a progression.

However, the rescaled adjacency matrix of the Fano plane, to be discussed in the next section, is circulant almost Hadamard, but does not have this property. The problem of finding the correct extension of the circulant Hadamard conjecture to the almost Hadamard matrix case is a quite subtle one, that we would like to raise here.

\section{The two-entry case}

In this section we study the almost Hadamard matrices having only two entries, $H\in M_N(x,y)$, with $x,y\in\mathbb R$. As a first remark, the usual Hadamard matrices $H\in M_N(\pm 1)$ are of this form. However, when trying to build a combinatorial hierarchy of the two-entry almost Hadamard matrices, the usual Hadamard matrices stand on top, and there is not so much general theory that can be developed, as to cover them.

We will therefore restrict attention to the following special type of matrices:

\begin{definition}
An $(a,b,c)$ pattern is a matrix $M\in M_N(x,y)$, with $N=a+2b+c$, such that, in any two rows, the number of $x/y/x/y$ sitting below $x/x/y/y$ is $a/b/b/c$.
\end{definition}

In other words, given any two rows of our matrix, we are asking for the existence of a permutation of the columns such that these two rows become:
$$\begin{matrix}
x\ldots x&x\ldots x&y\ldots y&y\ldots y\\
\underbrace{x\ldots x}_a&\underbrace{y\ldots y}_b&\underbrace{x\ldots x}_b&\underbrace{y\ldots y}_c
\end{matrix}$$

Oberve that the Hadamard matrices do not come in general from patterns. However, there are many interesting examples of patterns coming from block designs \cite{cdi}, \cite{sti}:

\begin{definition}
A $(v,k,\lambda)$ symmetric balanced incomplete block design is a collection $B$ of subsets of a set $X$, called blocks, with the following properties:
\begin{enumerate}
\item $|X|=|B|=v$.
\item Each block contains exactly $k$ points from $X$.
\item Each pair of distinct points is contained in exactly $\lambda$ blocks of $B$.
\end{enumerate} 
\end{definition}

The incidence matrix of a such block design is the $v\times v$ matrix defined by:
$$M_{bx}=\begin{cases}
1&\text{if }x\in b\\
0&\text{if }x\notin b
\end{cases}$$

The connection between designs and patterns comes from:

\begin{proposition}
If $N=a+2b+c$ then the adjacency matrix of any $(N,a+b,a)$ symmetric balanced incomplete block design is an $(a,b,c)$ pattern.
\end{proposition}

\begin{proof}
Indeed, let us replace the $0-1$ values in the adjacency matrix $M$ by abstract $x-y$ values. Then each row of $M$ contains $a+b$ copies of $x$ and $b+c$ copies of $y$, and since every pair of distinct blocks intersect in exactly $a$ points, cf. \cite{sti}, we see that every pair of rows has exactly $a$ variables $x$ in matching positions, so that $M$ is an $(a,b,c)$ pattern.
\end{proof}

\begin{figure}[htbp]
\centerline{\hbox{\epsfig{figure=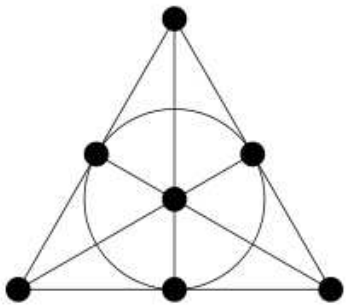,width=5.0cm}}}
\caption{The Fano plane}   
\end{figure}

As a first example, consider the Fano plane. The sets $X,B$ of points and lines form a $(7,3,1)$ block design, corresponding to the following $(1,2,2)$ pattern:
$$I_7=\begin{pmatrix}
x&x&y&y&y&x&y\\
y&x&x&y&y&y&x\\
x&y&x&x&y&y&y\\
y&x&y&x&x&y&y\\
y&y&x&y&x&x&y\\
y&y&y&x&y&x&x\\
x&y&y&y&x&y&x
\end{pmatrix}$$

Now remember that the Fano plane is the projective plane over $\mathbb F_2=\{0,1\}$. The same method works with $\mathbb F_2$ replaced by an arbitrary finite field $\mathbb F_q$, and we get:

\begin{proposition}
Assume that $q=p^k$ is a prime power. Then the point-line incidence matrix of the projective plane over $\mathbb F_q$ is a $(1,q,q^2-q)$ pattern.
\end{proposition}

\begin{proof}
The sets $X,B$ of points and lines of the projective plane over $\mathbb F_q$ are known to form a $(q^2+q+1,q+1,1)$ block design, and this gives the result.
\end{proof}

There are many other interesting examples of symmetric balanced incomplete block designs, all giving rise to patterns, via Proposition 3.3. For instance the famous Paley biplane \cite{bro}, pictured below, is a $(11,5,2)$ block design, and hence gives rise to a $(2,3,3)$ pattern. When assigning certain special values to the parameters $x,y$ we obtain a $11\times 11$ almost Hadamard matrix, that we believe to be optimal. See section 4 below.

\begin{figure}[htbp]
\centerline{\hbox{\epsfig{figure=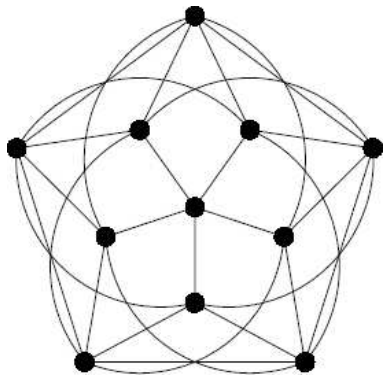,width=5.0cm}}}
\caption{The Paley biplane}   
\end{figure}

We consider now the problem of associating real values to the symbols $x,y$ in an $(a,b,c)$ pattern such that the resulting matrix $U(x,y)$ is orthogonal.

\begin{lemma}
Given $a,b,c\in\mathbb N$, there exists an orthogonal matrix having pattern $(a,b,c)$ iff $b^2\geq ac$. In this case the solutions are $U(x,y)$ and $-U(x,y)$, where
$$x=-\frac{t}{\sqrt{b}(t+1)},\quad\quad y=\frac{1}{\sqrt{b}(t+1)}$$
where $t=(b\pm\sqrt{b^2-ac})/a$ can be any of the solutions of $at^2-2bt+c=0$.
\end{lemma}

\begin{proof}
First, in order for $U$ to be orthogonal, the following conditions must be satisfied:
$$ax^2+2bxy+cy^2=0,\quad (a+b)x^2+(b+c)y^2=1$$

The first condition, coming from the orthogonality of rows, tells us that $t=-x/y$ must be the variable in the statement. As for the second condition, this becomes:
$$y^2=\frac{1}{(a+b)t^2+(b+c)}=\frac{1}{(at^2+c)+(bt^2+b)}=\frac{1}{2bt+bt^2+b}=\frac{1}{b(t+1)^2}$$

This gives the above formula of $y$, and hence the formula of $x=-ty$ as well.
\end{proof}

\begin{lemma}
Let $U=U(x,y)$ be orthogonal, corresponding to an $(a,b,c)$ pattern. Then $H=\sqrt{N}U$ is almost Hadamard iff $(N(a-b)+2b)|x|+(N(c-b)+2b)|y|\geq 0$.
\end{lemma}

\begin{proof}
We use the criterion in Definition 1.4 (2). So, let $S_{ij}={\rm sgn}(U_{ij})$. Since any row of $U$ consists of $a+b$ copies of $x$ and $b+c$ copies of $y$, we have:
$$(SU^t)_{ii}=\sum_k{\rm sgn}(U_{ik})U_{ik}=(a+b)|x|+(b+c)|y|$$

Regarding now $(SU^t)_{ij}$ with $i\neq j$, we can assume in the computation that the $i$-th and $j$-th row of $U$ are exactly those pictured after Definition 3.1 above. Thus:
\begin{eqnarray*}
(SU^t)_{ij}
&=&\sum_k{\rm sgn}(U_{ik})U_{jk}\\
&=&a\,{\rm sgn}(x)x+b\,{\rm sgn}(x)y+b\,{\rm sgn}(y)x+c\,{\rm sgn}(y)y\\
&=&a|x|-b|y|-b|x|+c|y|\\
&=&(a-b)|x|+(c-b)|y|
\end{eqnarray*}

We obtain the following formula for the matrix $SU^t$ itself:
\begin{eqnarray*}
SU^t
&=&2b(|x|+|y|)1_N+((a-b)|x|+(c-b)|y|)NJ_N\\
&=&2b(|x|+|y|)(1_N-J_N)+((N(a-b)+2b)|x|+(N(c-b)+2b)|y|))J_N
\end{eqnarray*}

Now since the matrices $1_N-J_N,J_N$ are orthogonal projections, we have $SU^t>0$ if and only if the coefficients of these matrices in the above expression are both positive. Since the coefficient of $1_N-J_N$ is clearly positive, the condition left is:
$$(N(a-b)+2b)|x|+(N(c-b)+2b)|y|\geq 0$$

So, we have obtained the condition in the statement, and we are done.
\end{proof}

\begin{proposition}
Assume that $a,b,c\in\mathbb N$ satisfy $c\geq a$ and $b(b-1)=ac$, and consider the $(a,b,c)$ pattern $U=U(x,y)$, where:
$$x=\frac{a+(1-a-b)\sqrt{b}}{Na},\quad y=\frac{b+(a+b)\sqrt{b}}{Nb}$$
Then $H=\sqrt{N}U$ is an almost Hadamard matrix. 
\end{proposition}

\begin{proof}
We have $b^2-ac=b$, so Lemma 2.5 applies, and shows that with $t=(b-\sqrt{b})/a$ we have an orthogonal matrix $U=U(x,y)$, where:
$$x=-\frac{t}{\sqrt{b}(t+1)},\quad y=\frac{1}{\sqrt{b}(t+1)}$$

In order to compute these variables, we use the following formula:
$$(a+b)^2-b=a^2+b^2+2ab-b=a^2+2ab+ac=Na$$

This gives indeed the formula of $y$ in the statement:
$$y=\frac{a}{(a+b)\sqrt{b}-b}=\frac{(a+b)\sqrt{b}+b}{Nb}$$

As for the formula of $x$, we can obtain it as follows:
$$x=-ty=\frac{(\sqrt{b}-b)((a+b)\sqrt{b}+b)}{Nab}=\frac{a+(1-a-b)\sqrt{b}}{Na}$$

Let us compute now the quantity appearing in Lemma 3.6. We have:
\begin{eqnarray*}
N(a-b)+2b
&=&(a+2b+c)(a-b)+2b\\
&=&a^2+ab-2b^2+ac-bc+2b\\
&=&a^2+ab-ac-bc\\
&=&(a-c)(a+b)
\end{eqnarray*}

Similarly, $N(c-b)+2b=(c-a)(c+b)$, so the quantity in Lemma 3.6 is $Ky$, with:
\begin{eqnarray*}
K
&=&(a-c)(a+b)t+(c-a)(c+b)\\
&=&(c-a)(c+b-(a+b)t)\\
&=&\frac{c-a}{a}(ac+ab-(a+b)(b-\sqrt{b}))\\
&=&\frac{c-a}{a}((ac-b^2)+(a+b)\sqrt{b})\\
&=&\frac{c-a}{a}((a+b)\sqrt{b}-b)
\end{eqnarray*}

Since this quantity is positive, Lemma 3.6 applies and gives the result.
\end{proof}

\begin{theorem}
Assume that $q=p^k$ is a prime power. Then the matrix $I_N\in M_N(x,y)$, where $N=q^2+q+1$ and
$$x=\frac{1-q\sqrt{q}}{\sqrt{N}},\quad y=\frac{q+(q+1)\sqrt{q}}{q\sqrt{N}}$$
having $(1,q,q^2-q)$ pattern coming from the point-line incidence of the projective plane over $\mathbb F_q$ is an almost Hadamard matrix.
\end{theorem}

\begin{proof}
Indeed, the conditions $c\geq a$ and $b(b-1)=ac$ needed in Proposition 3.7 are satisfied, and the variables constructed there are $x'=x/\sqrt{N}$ and $y'=y/\sqrt{N}$.
\end{proof}

There are of course many other interesting examples of two-entry almost Hadamard matrices, all worth investigating in detail, but we will stop here. Indeed, the main purpose of the reminder of this paper is to provide a list of almost Hadamard matrices which are ``as optimal as possible'', at $N=2,3,\ldots,13$, and in order to establish this list, we will just need the incidence matrices $I_N$, plus the matrix coming from the Paley biplane.

Let us mention however two more important aspects of the general theory:
\begin{enumerate}
\item The series $I_N$ is the particular case of a 2-parameter series $I_N^{(d)}$. Indeed, associated to $q=p^k$ and $d\in\mathbb N$ is a certain $([d+2]_q,[d+1]_q,[d]_q)$ block design coming from $\mathbb F_q$, where $[e]_q=(q^e-1)/(q-1)$, cf. \cite{cdi}, \cite{sti}. Thus by Proposition 3.7 we obtain an almost Hadamard matrix $I_N^{(d)}$, having pattern $(\frac{q^d-1}{q-1},q^d,q^d(q-1))$.

\item Trying to find block designs and patterns leads to the following chess problem: consider a $N\times N$ chessboard, take $NM$ rooks with $M\leq N/2$, fix an integer $K\leq M$, and try to place all the rooks on the board such that: (a) there are exactly $M$ rooks on each row and each column of the board, and (b) for any pair of rows or columns, there are exactly $K$ pairs of mutually attacking rooks.
\end{enumerate}

Finally, let us mention that there are many questions raised by the almost Hadamard matrices, at the quantum algebraic level. For instance the symmetries of Hadamard matrices are known to be described by quantum permutations \cite{ba2}, and it would be interesting to have a similar result for the almost Hadamard matrices. This might probably bring some new ideas on the ``homogeneous implies quantum homogeneous'' question, raised in \cite{ba1}, and having connections with the finite projective planes. Also, an interesting link between mutually unbiased bases, complex Hadamard matrices and affine planes was emphasized in \cite{ber}, but its relation with our present investigations is not known yet.

\section{List of examples}

In this section we present a list of examples of almost Hadamard matrices, for small values of $N$. Since we are mainly interested in the optimal case, our examples will be chosen to be ``as optimal as possible'', i.e. will be chosen as to have big $1$-norms.

So, let us first compute the 1-norms for the examples that we have. In what follows $L_N$ is the matrix found in Theorem 2.6, and $I_N$ is the matrix found in Theorem 3.8.

\begin{theorem}
The $1$-norms of the basic examples of almost Hadamard matrices are:
\begin{enumerate}
\item Hadamard case: $||H/\sqrt{N}||_1=N\sqrt{N}$, for any $H\in M_N(\pm 1)$ Hadamard.

\item Basic series case: $||K_N/\sqrt{N}||_1=3N-4$, where $K_N=\sqrt{N}(2J_N-1_N)$.

\item Circulant series case: $||L_N/\sqrt{N}||_1=\frac{2}{\pi}N\log N+O(N)$.

\item Incidence series case: $||I_N/\sqrt{N}||_1=(q^2-q-1)+2q(q+1)\sqrt{q}$.
\end{enumerate}
\end{theorem}

\begin{proof}
The first two assertions are clear. For the third one, with $N=2n+1$ we have:
$$||L_N/\sqrt{N}||_1=2\sum_{i=0}^n\cos^{-1}\frac{i\pi}{N}+O(N)=2\sum_{i=0}^n\sin^{-1}\frac{(2i+1)\pi}{2N}+O(N)$$

Now by using $\sin x\sim x$ and $\sum_{i=1}^k1/i=\log k+O(1)$ we obtain, as claimed:
$$||L_N/\sqrt{N}||_1=\frac{4N}{\pi}\sum_{i=0}^n\frac{1}{2i+1}+O(N)=\frac{2N}{\pi}\log N+O(N)$$

As for the last assertion, let first $U$ be the matrix in Lemma 3.6. We have:
$$||U||_1=N((a+b)|x|+(b+c)|y|)$$

In the particular case of the orthogonal matrices in Proposition 3.7, we get:
$$||U||_1=(c-a)+\frac{(a+b)(2a+2b-2)}{a}\sqrt{b}$$

Now with $a=1$, $b=q$, $c=q^2-q$, this gives the formula in the statement.
\end{proof}

Observe that at $N$ big the above matrices $K_N,L_N,I_N$ are far from being optimal. In fact, with $N\to\infty$, the corresponding orthogonal matrices don't even match the average of the 1-norm on $O(N)$, which, according to \cite{bc1}, is $\sim cN\sqrt{N}$, with $c=0.797..$

With these ingredients in hand, let us discuss now the various examples: 

At $N=2$ we have the Walsh matrix $H_2$, which is of course optimal.

At $N=3$ we have the almost Hadamard matrix $K_3$, which is optimal:
$$K_3=\frac{1}{\sqrt{3}}\begin{pmatrix}-1&2&2\\ 2&-1&2\\ 2&2&-1\end{pmatrix}$$ 

At $N=4$ we have the Hadamard matrix $H_4\sim K_4$, once again optimal.

At $N=5$ we have the matrix $K_5$, that we believe to be optimal as well:
$$K_5=\frac{1}{\sqrt{5}}\begin{pmatrix}-3&2&2&2&2\\ 2&-3&2&2&2\\ 2&2&-3&2&2\\ 2&2&2&-3&2\\ 2&2&2&2&-3\end{pmatrix}$$

At $N=6$ it is plausible that the optimal AHM is simply $K_3\otimes H_2$:
$$K_3\otimes H_2=\frac{1}{\sqrt{3}}\begin{pmatrix}-1&2&2&-1&2&2\\ 2&-1&2&2&-1&2\\ 2&2&-1&2&2&-1\\ -1&2&2&1&-2&-2\\ 2&-1&2&-2&1&-2\\ 2&2&-1&-2&-2&1\end{pmatrix}$$

At $N=7$ we have the incidence matrix of the Fano plane ($x=2-4\sqrt{2}$, $y=2+3\sqrt{2}$):
$$I_7=\frac{1}{2\sqrt{7}}\begin{pmatrix}
x&x&y&y&y&x&y\\
y&x&x&y&y&y&x\\
x&y&x&x&y&y&y\\
y&x&y&x&x&y&y\\
y&y&x&y&x&x&y\\
y&y&y&x&y&x&x\\
x&y&y&y&x&y&x
\end{pmatrix}$$

At $N=8$ we have the third Walsh matrix $H_8=H_2\otimes H_4$, of course optimal.

At $N=9$ we just have the matrix $K_3\otimes K_3$, which can be shown not to be optimal. 

At $N=10$ we believe that the matrix $K_5\otimes H_2$ is optimal:
$$K_5\otimes H_2=\frac{1}{\sqrt{5}}\begin{pmatrix}
-3&2&2&2&2&-3&2&2&2&2\\ 
2&-3&2&2&2&2&-3&2&2&2\\ 
2&2&-3&2&2&2&2&-3&2&2\\ 
2&2&2&-3&2&2&2&2&-3&2\\ 
2&2&2&2&-3&2&2&2&2&-3\\
-3&2&2&2&2&3&-2&-2&-2&-2\\ 
2&-3&2&2&2&-2&3&-2&-2&-2\\ 
2&2&-3&2&2&-2&-2&3&-2&-2\\ 
2&2&2&-3&2&-2&-2&-2&3&-2\\ 
2&2&2&2&-3&-2&-2&-2&-2&3
\end{pmatrix}$$

At $N=11$ we have the matrix of the Paley biplane ($x=6-12\sqrt{3}$, $y=6+10\sqrt{3}$):
$$P_{11}=\frac{1}{6\sqrt{11}}\left(\begin{array}{cccccccccccccc}
y& x& y& x& x& x& y& y& y& x& y\\
y& y& x& y& x& x& x& y& y& y& x\\
x& y& y& x& y& x& x& x& y& y& y\\
y& x& y& y& x& y& x& x& x& y& y\\
y& y& x& y& y& x& y& x& x& x& y\\
y& y& y& x& y& y& x& y& x& x& x\\
x& y& y& y& x& y& y& x& y& x& x\\
x& x& y& y& y& x& y& y& x& y& x\\
x& x& x& y& y& y& x& y& y& x& y\\
y& x& x& x& y& y& y& x& y& y& x\\
x& y& x& x& x& y& y& y& x& y& y
\end{array}\right)$$

At $N=12$ we have the Sylvester Hadamard matrix $S_{12}$, of course optimal.

At $N=13$ we have the incidence matrix of $P(\mathbb F_3)$ ($x=3-9\sqrt{3}$, $y=3+4\sqrt{3}$):
$$I_{13}=\frac{1}{3\sqrt{13}}\left(\begin{array}{cccccccccccccc}
x&x&x&x&y&y&y&y&y&y&y&y&y\\
x&y&y&y&x&x&x&y&y&y&y&y&y\\
x&y&y&y&y&y&y&x&x&x&y&y&y\\
x&y&y&y&y&y&y&y&y&y&x&x&x\\
y&x&y&y&y&y&x&y&x&y&y&y&x\\
y&x&y&y&y&x&y&y&y&x&x&y&y\\
y&x&y&y&x&y&y&x&y&y&y&x&y\\
y&y&x&y&y&y&x&x&y&y&x&y&y\\
y&y&x&y&x&y&y&y&y&x&y&y&x\\
y&y&x&y&y&x&y&y&x&y&y&x&y\\
y&y&y&x&y&x&y&x&y&y&y&y&x\\
y&y&y&x&y&y&x&y&y&x&y&x&y\\
y&y&y&x&x&y&y&y&x&y&x&y&y
\end{array}\right)$$

The norms of the above matrices can be computed by using the various formulae in Theorem 4.1 and its proof, and the results are summarized in Table 1.

\begin{table}
\begin{tabular}{ccccccc}
\hline 
$N$&matrix&1-norm (formula)&1-norm (numeric)&$N\sqrt{N}$&remarks
\tabularnewline\hline\hline
2&$H_2$&$2\sqrt{2}$&2.828&2.828&Hadamard
\tabularnewline\hline 
3&$K_3$&5&5.000&5.196&optimal 
\tabularnewline\hline 
4&$K_4$&8&8.000&8.000&Hadamard 
\tabularnewline\hline 
5&$K_5$&11&11.000&11.118& 
\tabularnewline\hline 
6&$K_3\otimes H_2$&$10\sqrt{2}$8&14.142&14.697& 
\tabularnewline\hline 
7&$I_7$&$1+12\sqrt{2}$&17.971&18.520
\tabularnewline\hline 
8&$H_8$&$16\sqrt{2}$&22.627&22.627&Hadamard 
\tabularnewline\hline 
9&?&--&$>26.513$&27.000&  
\tabularnewline\hline 
10&$K_5\otimes H_2$&$22\sqrt{2}$&31.113&31.623&
\tabularnewline\hline 
11&$P_{11}$&$1+20\sqrt{3}$&35.641&36.483&
\tabularnewline\hline 
12&$S_{12}$&$24\sqrt{3}$&41.569&41.569&Hadamard
\tabularnewline\hline 
13&$I_{13}$&$5+24\sqrt{3}$&46.569&46.872&
\tabularnewline\hline\hline
\end{tabular}
\bigskip

\caption{Almost Hadamard matrices $H\in M_N(\mathbb R)$, chosen as for the corresponding orthogonal matrices $U=H/\sqrt{N}$ to have big 1-norm. All matrices are believed to be optimal. The lower bound for the maximum of the 1-norm on $O(9)$, which is not optimal, was obtained by numerical simulation.}
\label{table1}
\end{table}

\section{Conclusion}

We have seen in this paper that the Hadamard matrices are quite nicely generalized by the almost Hadamard matrices (AHM), which exist at any given order $N\in\mathbb N$. Our study of these matrices, which was for the most of algebraic nature, turns to be related to several interesting combinatorial problems, notably to the Circulant Hadamard Conjecture. 

We believe that the AHM can be used as well in connection with several problems in quantum physics, in a way somehow similar to the way the complex Hadamard matrices (CHM) are used. Indeed, since the CHM exist as well at any given order $N\in\mathbb N$, these matrices proved to be useful in several branches of quantum physics. For instance in quantum optics they are sometimes called the ``Zeilinger matrices'', as they can be applied to design symmetric linear multiports, used to split the beam into $N$ parts of the same intensity and to analyze interference effect \cite{rz+}, \cite{jsz}. In the theory of quantum information one uses quantum Hadamard matrices to construct mutually unbiased bases (MUB) \cite{iva}, \cite{de+} to design teleportation and dense coding schemes. As shown in the seminal work of Werner \cite{wer} these two problems are in fact equivalent and also equivalent to construction of unitary depolarisers and maximally entangled bases \cite{wgc}.

Although quantum mechanics in a natural way relays on a complex Hilbert space it is often convenient to study a simplified problem and restrict attention to the subset of real quantum states only. Such an approach can be useful in theoretical investigations of quantum entanglement \cite{cfr} or also in experimental studies on engineering of quantum states, as creating of a real state by an orthogonal rotation usually requires less effort than construction of an arbitrary complex state. 

There exists therefore a natural motivation to ask similar problems concerning e.g. unbiased bases and teleportation schemes in the real setup. For instance, it is known that for any $N$ there exist $\leq N/2+1$ real MUB and for most dimensions their actual number is not larger than 3, cf. \cite{bs+}, while for any prime $N$ there exist $N+1$ complex MUB. Note that the real MUB and Hadamard matrices are closely related to several combinatorial problems \cite{hko}, \cite{lmo}, \cite{mrw}. In the case the maximal number of real MUB does not exist one can search for an optimal set of real bases which are approximately unbiased.

Let us now return to the construction of symmetric multiports which relay on the complex Hadamard matrices \cite{rz+}, \cite{jsz}. In the case of a $k$-qubit system there exist Hadamard matrices of order $N=2^k$, so one may use real orthogonal matrices for this purpose. However, already for $N=3,5,6,7$ real Hadamard matrices do not exist, so in these dimensions, if one restricts the rotations to orthogonal matrices, there are no real symmetric multiports. Therefore, for these dimensions one may always raise the following question: what is the optimal, approximate solution of the problem, if one is allowed to use only real states and orthogonal matrices? The almost Hadamard matrices analyzed in this paper are directly applicable for such a class of problems.

\end{document}